\documentclass[reqno]{amsart}

\usepackage[abbrev]{amsrefs}		
\usepackage{hyperref}
\usepackage{mathrsfs, cancel,stmaryrd}
\usepackage{listings,color,xcolor}
\usepackage{xargs}

\def\Ratl{\mathbb{Q}}
\def\fin{\mathbb{F}}
\def\reR{\mathbb{R}}

\def\Proj{\mathbb{P}}
\def\cxC{\mathbb{C}}
\def\isom{\cong}
\def\intN{\mathbb{N}}
\def\intZ{\mathbb{Z}}
\def\spmode{\scriptstyle}
\def\goto{\rightarrow}
\def\d{{\text{d}}}
\def\and{\text{ and }}

\theoremstyle{plain}
\newtheorem{defi}{Definition}[section]
\newtheorem{lemma}[defi]{Lemma}

\newtheorem{thm}[defi]{Theorem}

\newtheorem{notation}[defi]{Notation}

\theoremstyle{remark}
\newtheorem{rmk}{Remark}

\title{Rational Points on Rational Curves}
\author{Brecken Beers}
\email{brecken.beers@aggiemail.usu.edu}
\author{Yih Sung}
\email{yih.sung@usu.edu}

\begin{document}

\begin{abstract}
For a given elliptic curve, its associated $L$-function evaluated at $1$ is closely related to its real period. In this article, we generalize this principle to a rational curve. We count the rational points over all finite fields and use all the counting information to define two $L$- type series. Then we consider special values of these series at $1$. One of the $L$-type series matches the Dirichlet $L$-series of modulo $4$, so the evaluation at $1$ is $\pi/4$; the special evaluation at $1$ of the other $L$-type series is equal to a real period associated to the rational curve. This identity confirms the general principle that an $L$-type series associated to a variety can reflect its geometry.
\end{abstract}


\maketitle


\section{Introduction}
\subsection{$L$-function of A Smooth Project Curve}
Let $C$ be be a smooth projective curve over $\Ratl$ of genus $g$. Classically, its associated $L$-function is defined by the Euler product $L(C,s) = \prod_p L_p(p^{-s})^{-1}$, where $L_p(T)$ is the local $L$-function. At the good primes, $C$ is defined and nonsingular over $\fin_p$, and $L_p(T)$ is defined by the equation
$$
Z(C/\fin_p;T) = \exp \Big( \sum_{k=1}^\infty |C(\fin_p)| T^k/k \Big) = \frac{L_p(T)}{(1-T)(1-pT)};
$$
at the bad primes, we refer to \cite{ref_bouw&wewers_17,ref_serre70}. In particular, when $C=E$ is an elliptic curve over $\Ratl$ of conduct $N$, its associated $L$-series is defined as follows.
\begin{defi}[$L$-Function of An Elliptic Curve over $\Ratl$, \cite{ref_arithmetic_of_ellip_curve}*{Sec A.16}]
$$
L(E,s) := \prod_{p:\text{prime}} L_p(E,s)^{-1}
$$
where
$$
L_p(E,s) =
\begin{cases}
(1 - a_p p^{-s} + p^{1-2s}) & \text{if $p\not\mid N$} \\
(1 - a_p p^{-s}) & \text{if $p\mid N$ and $p^2\not\mid N$} \\
1 & \text{if $p^2\not\mid N$}
\end{cases}
$$
and $a_p = p+1-|E_p|$ if $p\not\mid N$; $a_p=\pm 1$ if $p\mid N$ and $p^2\not\mid N$.
\end{defi}

Geometrically, one can consider the \emph{N\'eron period}
\begin{equation}
\Omega(E) := \int_{E(\reR)} \omega_E \in\reR,
\end{equation}
where $\omega_E=\d x/2\d y$ is the N\'eron differential and $E(\reR)$ refers to the real points of $E$. Surprisingly, Birch's theorem \cite{ref_birch71} describes a relation between the special value of the $L$-function and N\'eron period:
\begin{equation} \label{eq:ellip_priod_counting}
\frac{L(E,1)}{\Omega(E)} \in \Ratl,
\end{equation}
which suggests one can recover the "length" of the curve by carefully arranging the arithmetic counting. Inspired by \eqref{eq:ellip_priod_counting}, we want to construct an $L$-type series on a rational curve $M$ and relate the series to a right notion of length on $M$. Unlike an elliptic curve, there is no global differential on $\Proj^1$. As we only require an integral over the real points of $M$, we do not need a globally defined differential. Instead, we consider a differential defined over the affine part of $M$. Hence we can make sense of the integral over a real period on the affine part of $M$. On the arithmetic side, the Zeta function of $\Proj^n$ in the classical approach is 
$$
Z(\Proj^n/\fin_q;T) = \frac{1}{(1-T)(1-qT)\cdots(1-q^nT)}
$$ 
\cite{ref_arithmetic_of_ellip_curve}*{Ch 5,~Example 2.1}, which implies the $L$-function of $\Proj^n$ is 
\begin{equation} \label{eq:L_func_of_P^1}
L(\Proj^n,s) = 1,
\end{equation}
which does not reflect any topological information about $M$. In order to extract useful arithmetic information about $M$, we turn to count on $M_{\rm aff}$, the affine part of $M$. In this article, we will present a new way to construct an $L$-type series on $M_{\rm aff}$ to capture the topological features of $M$ such as its real period. In the end, we will link the geometry of $M$ to the arithmetic of $M$ by a twisted $L$-series.


\subsection{Statement of Results}
Let us state our main results. Let $M$ be defined by
\begin{equation}
X^2+Y^2 = Z^2
\end{equation}
in $\Proj^2$. In particular, the equation reduces to 
$$
x^2+y^2=1
$$
on $M_{\rm aff}=\cxC^2$, the affine part of $M$. Let $p$ be a prime and $M_{\rm{aff},p}$ be the affine variety defined over $\fin_p$. We can then write $|M_{\rm{aff},p}|=p-a_p$, where $a_p$ is called the \emph{error term} in our counting because we expect to have $p$ points on $\Proj^1_{{\rm aff}, p}$. Define $\alpha_p=(-1)^p a_p$ and we obtain two multiplicative characters due to prime factorization. Denote $n=p_1^{r_1}\cdots p_k^{r_k}$ and define
$$
\chi(n) = \prod_{1\le i\le k} a_{p_i^{r_i}} \;\and\; \hat\chi(n) = \prod_{1\le i\le k} \alpha_{p_i^{r_i}}.
$$
Then, we can construct two Dirichlet series associated to $\chi$ and $\hat \chi$:
\begin{gather*}
\zeta(M,s) = \sum_{k=1}^\infty \frac{a_{2k-1}}{(2k-1)^s}
= \sum_{k=1}^\infty \frac{(-1)^{2k-1}}{(2k-1)^s}, \\
\hat\zeta(M,s) = \sum_{k=1}^\infty \frac{\alpha_{2k-1}}{(2k-1)^s}.
\end{gather*}
In fact, there is an equation to link these two $L$-series. We will prove the following:
\begin{thm}[Functional Equation of $L$-series] \label{thm:eq_of_L_series}
$$
\zeta(M,s) \cdot \hat \zeta(M,s) = L_{\chi_2}(X,2s),
$$
where $\chi_2$ is the Dirichlet character of modulo 2. In other words,
$$
\sum_{k=1}^\infty \frac{(-1)^{2k-1}}{(2k-1)^s} \cdot \sum_{k=1}^\infty \frac{\alpha_{2k-1}}{(2k-1)^s} 
= \sum_{k=1}^\infty \frac{1}{(2k-1)^{2s}},
$$
Particularly, when $s=1$,
$$
\sum_{k=1}^\infty \frac{(-1)^{2k-1}}{2k-1} \cdot \sum_{k=1}^\infty \frac{\alpha_{2k-1}}{2k-1} 
= \sum_{k=1}^\infty \frac{1}{(2k-1)^2}.
$$
\end{thm}

On the geometric side, over $M_{\rm aff}$, there is a well defined holomorphic differential form
$$
\frac{\d x}{y} = -\frac{\d y}{x}.
$$
Naturally, the integral of $\d x/y$ over the real cycle, namely, the real solutions of the equation $x^2+y^2=1$, represents the circumference of the circle. By taking the special value of the $L$-series $\hat\zeta$ at $1$, we actually recover the quarter circumference.
\begin{thm}[Birch Type Theorem of A Rational Curve] \label{thm:main_result}
\begin{equation*}
\zeta(M,1) = \frac{\pi}{4} \;\and\;
\hat\zeta(M,1) = \frac{1}{2} \int_{-1}^1 \frac{1}{\sqrt{1-x^2}} \, \d x
= \int_{0}^1 \frac{1}{\sqrt{1-x^2}} \, \d x = \frac{\pi}{2}.
\end{equation*}
\end{thm}
In summary, we discover a new type of $L$-series on $M_{\rm aff}$, and we can recover the quarter length of the real circle by evaluating it at special values. Furthermore, we not only know $\hat\zeta(M,1)=\int_0^1 \frac{1}{\sqrt{1-x^2}}\,\d x=\pi/2$ numerically, we also find a combinatorial argument to transfer the integral $\int_{0}^1 \frac{1}{\sqrt{1-x^2}} \, \d x$ into the twisted $L$-series $\hat\zeta(M,1)$.


\subsection{Contents}
This article is structured as follows: in Section \ref{sec:count_ratl_pt_on_ratl_curve}, we calculate the number of rational points of $M$ over the finite field $\fin_q$, where $q=p^n$. Through detailed calculation, we elaborate why the classical $L$-function of a rational curve is trivial. In Section \ref{sec:new_type_of_L_series}, we construct a classical Dirichlet $L$-series and a twisted $L$-series by the error terms in rational point counting developed in Section \ref{sec:count_ratl_pt_on_ratl_curve}. Then we consider special values of the $L$-series and prove Theorem \ref{thm:eq_of_L_series}, which is a functional equation to connect these $L$-series. In section \ref{sec:main_result}, we provide a combinatorial argument to prove Theorem \ref{thm:main_result}.


\section{Rational Point Counting on Rational Curves} \label{sec:count_ratl_pt_on_ratl_curve}
\subsection{Point Counting}
In this section we are going to investigate the rational points on a rational curve. Let $M\isom\Proj^1$ defined as $x^2+y^2=z^2$ in $\Proj_\cxC^2$ and its affine part $M_{\rm aff}$ is defined by $x^2+y^2=1$ in $\cxC^2$. 
\begin{notation}
Let $p$ be a prime and $q=p^n$ for some integer $n\in\intN$. We denote $M_q$ the variety defined over $\fin_q$ and $|M_q|$ the number of rational points over $\fin_q$.
\end{notation}
Let us count rational points of the affine part of $M$ over the finite field $\fin_{2^n}$.
\begin{lemma}\label{lem:count_over_2^n}
$$
|M_{\rm{aff},2^n}| = 2^n.
$$
\end{lemma}
\begin{proof}
As the characteristic of $\fin_{2^n}$ is $2$, 
$$
y^2 = 1 - x^2 = 1+x^2 = (1+x)^2
$$
in $\fin_{2^n}$. Hence, $y=\pm (1+x)$. However, in $\fin_{2^n}$, $(1+x)=-(1+x)$. In other words, for every $x$, there exists a unique solution to the equation $y=1+x$. Therefore, $|M_{\rm{aff},2^n}| = |\fin_{2^n}| = 2^n$. In summary, there is no error term in counting over $\fin_{2^n}$.
\end{proof}

Let us count rational points of $M_q$, where $q=p^n$ and $p$ is an odd prime.
\begin{lemma}\label{lem:count_over_p^n}
Let $p$ be an odd prime and $q=p^n$ for some $n$. Then
\begin{equation} \label{eq:pt_count_on_aff_part}
|M_{\rm{aff},q}| = q - \underbrace{(-1)^{\frac{q+1}{2}}}_{\text{err. term}}. 
\end{equation}
\end{lemma}
\begin{proof}
For the finite field $\fin_p$, the result is equivalent to solving the equation $y^2=1-x^2$. As the set $S=\{ y\in\fin_p^* \mid y=x^2\;\text{for some $x$}\}$ has $\frac{p-1}{2}$ elements, the set $S_1=\{(1-x^2)\neq 0\mid x\in\fin_p\}$ has at most $\frac{p-1}{2}$ elements. Since we can take $\pm y$ for each $x\in S_1$, there are at most $\frac{p-1}{2}\cdot 2$ solutions for $y\neq 0$. If $x=\pm 1$, we have $y=0$. Hence, there are at most $\frac{p-1}{2}\cdot 2 + 2=p+1$ solutions and at least $2$ solutions. 

On the other hand, we can calculate $|M_p|$ modulo $p$ by the standard trick (cf. \cite{ref_yih17}):
\begin{align*}
|M_{\rm{aff},p}| &\equiv \sum_{x\in \fin_p} (1-x^2)^{\frac{p-1}{2}}
= \sum_{x\in \fin_p} \sum_{k=0}^{\frac{p-1}{2}} C^{\frac{p-1}{2}}_k (-1)^k x^{2k} \\
&= \sum_{k=0}^{\frac{p-1}{2}} (-1)^k C^{\frac{p-1}{2}}_k \Big( \sum_{x\in \fin_p} x^{2k}   \Big)
\equiv -(-1)^{\frac{p-1}{2}}
\end{align*}
because
\begin{equation} \label{coeff_of_counting_pts} 
\sum_{x\in\mathbb{F}_p} x^k \equiv
\begin{cases}
-1, & \text{if $(p-1)\mid k$} \\
0, & \text{if $(p-1)\!\not\mid \,k$}
\end{cases}
\mod p.
\end{equation}
As $2\le |M_{\rm{aff},p}| \le p+1$, we know $|M_{\rm{aff},p}|=p-(-1)^{\frac{p-1}{2}}$.

For the general $q=p^n$ case we refer to \cite[Thm 3.1]{ref_aabrandt&hansen_18}.
\end{proof}

We turn to count points at infinity. At infinity, the equation of $M$ locally can be represented by
$$
1+\frac{Y^2}{X^2} = \frac{Z^2}{X^2} \implies 1+y^2=z^2,
$$
so $M$ intersects with the $\Proj^1_\infty$ at the solutions of $1+y^2=0$. Over the finite field $\fin_{2^n}$, the equation is always solvable because
$$
y^2=-1=1 \implies y=\pm 1 = 1.
$$
In other words,
\begin{lemma}
\begin{equation}
|M_{\infty,2^n}| = 1.
\end{equation}
\end{lemma}
With regard to the general odd primes, we have:
\begin{lemma}
Let $p$ be an odd prime and $q=p^n$.
$$
|M_{\infty,p}| = 
\begin{cases}
2 & \text{if $p\equiv 1\mod 4$} \\
0 & \text{if $p\equiv 3\mod 4$}
\end{cases}.
$$
Therefore,
\begin{equation*}
|M_{\infty,p}| = 1+ (-1)^{\frac{p-1}{2}} = 1- \big(\underbrace{-(-1)^{\frac{p-1}{2}}}_{\text{err. term}} \big).
\end{equation*}
In general,
\begin{equation} \label{eq:pt_count_at_infty}
|M_{\infty,q}| = 1+ (-1)^{\frac{q-1}{2}} = 1- \big(\underbrace{-(-1)^{\frac{q-1}{2}}}_{\text{err. term}} \big).
\end{equation}
\end{lemma}
\begin{proof}
By Fermat's Little theorem, $a^{\frac{p-1}{2}}\equiv 1$ if and only if there exists $x\in\fin_p$ such that $x^2=a$. Therefore, the result follows. Regarding the general finite field $\fin_q$, we refer to \cite{ref_yih17}.
\end{proof}

\begin{rmk} \label{rmk:classical_L_func_of_Proj^1}
Combining \eqref{eq:pt_count_on_aff_part} and \eqref{eq:pt_count_at_infty}, we obtain
$$
|M_q| = |M_{\rm{aff},q}| + |M_{\infty,q}| = q+1,
$$
namely, there is \emph{no} error term in rational point counting on $\Proj^1$. That is why the classical $L$-function of $\Proj^1$ is trivial. It also suggests that we can derive \eqref{eq:pt_count_on_aff_part} by combing \eqref{eq:pt_count_at_infty} and \eqref{eq:L_func_of_P^1}. In this approach, we do not need to go through the technical calculation involving the Jacobi summations.
\end{rmk}


\section{New Type of $L$-series} \label{sec:new_type_of_L_series}
\subsection{Twisted $L$-series}
As Remark \ref{rmk:classical_L_func_of_Proj^1} points out, from the arithmetic perspective, there is no interesting $L$-function if we consider the entire compact Riemann surface; on the other hand, geometrically, we have to give up the $\infty$ of $\Proj^1$ and consider its affine part, which is the regular complex plane $\cxC$, in order to get a holomorphic differential. From now on, we will focus on the point counting and integrals over the affine part.

Similar to the $L$-function of elliptic curves, we will use the error terms in point counting to construct $L$-function type series. Let $p$ be an \emph{odd} prime. We denote 
$$
a_p := (-1)^{\frac{p-1}{2}},
$$ 
which is the error term in \eqref{eq:pt_count_on_aff_part}. We intend to define an $L$-function type generating function which makes $a_p$ and $a_{p^n}$ compatible. The key is the iteration equation of $a_p$.
\begin{lemma}\label{lem:a_p}
\begin{equation}\label{eq:a_p}
a_{p^k}=a_{p^{k-1}}\cdot a_p.
\end{equation}
Hence, $a_{p^k} = (a_p)^k$.
\end{lemma} 
\begin{proof}
The result is just a mater of calculation. 
\begin{align*}
a_{p^k}=(-1)^{\frac{p^k-1}{2}} = (-1)^{\frac{p^k-p+p-1}{2}} 
= \underbrace{\Big( (-1)^{\frac{p^{k-1}-1}{2}} \Big)^p}_{=(-1)^{\frac{p^{k-1}-1}{2}} } \cdot a^{\frac{p-1}{2}}.
\end{align*}
Thus, $a_{p^k}=a_{p^{k-1}}\cdot a_p$.
\end{proof}

In particular, consider $k=1$ and we obtain $a_p = a_{p^0} \cdot a_p = a_1 \cdot a_p$, which implies 
\begin{equation}
a_{p^0} = a_1 = 1.
\end{equation}
However, $a_{p^k}=(-1)^{\frac{p^k-1}{2}}$ is not the only solution to the recursive equation \eqref{eq:a_p}:
\begin{equation}
\alpha_{p^k} = (-1)^{k} (-1)^{\frac{p^k-1}{2}} = (-1)^{k} a_{p^k}
\end{equation}
also satisfies equation \eqref{eq:a_p}: 
$$
\alpha_{p^k}=(-1)^{k-1} (-1)^{\frac{p^{k-1}-1}{2}} \cdot (-1)(-1)^{\frac{p-1}{2}}
= \alpha_{p^{k-1}}\cdot \alpha_p.
$$ 
Hence, by Lemma \ref{lem:a_p}, 
\begin{equation}
\alpha_{p^k} = (\alpha_p)^k.
\end{equation}

Now we have two options to construct the desired series. First, we can construct a classical Dirichlet $L$-series by $\{a_n\}$. Let
$$
f_p(s) = \sum_{k=0}^\infty \frac{a_{p^k}}{p^{ks}} 
=  \sum_{k=0}^\infty \frac{(a_p)^k}{(p^s)^k} = \frac{1}{1-\frac{a_p}{p^s}},
$$
which is precisely the Dirichlet $L$-series of modulo $4$:
$$
\zeta(M,s) = \prod_{p} f_p(s) 
 = \prod_{p} \frac{1}{1-\frac{a_p}{p^s}}
 = \sum_{n=1}^\infty \frac{a_n}{n^s} = \sum_{n=0}^{\infty} \frac{(-1)^n}{(2n+1)^s}.
$$

We can twist the construction on the error terms because $\{\alpha_{p^n}\}$ also satisfies the iterative equation \eqref{eq:a_p}. Consider an $L$-type generating function at prime $p$ by the twisted factors:
\begin{align*}
\hat f_p(s) = \sum_{k=0}^\infty \frac{\alpha_{p^k}}{p^{ks}} 
=  \sum_{k=0}^\infty \frac{(-1)^k (a_p)^k}{(p^s)^k} = \frac{1}{1+\frac{a_p}{p^s}},
\end{align*}
which induces a Dirichlet $L$-type function
$$
\hat\zeta(M,s) = \prod_{p} \hat f_p(s) 
 = \prod_{p} \frac{1}{1+\frac{a_p}{p^s}}
 = \prod_{p} \frac{1}{1-\frac{\alpha_p}{p^s}} 
 = \sum_{n=1}^\infty \frac{\alpha_n}{n^s},
$$
where $\alpha_n$ satisfies the equations
\begin{equation}
\begin{cases}
\alpha_{p^k} = (-1)^{k} (-1)^{\frac{p^k-1}{2}} & \text{if $p$ is a prime} \\
\alpha_{rs} = \alpha_r \cdot \alpha_s & \text{if $(r,s)=1$}
\end{cases}.
\end{equation}

These $\{a_n\}$ and $\{\alpha_n\}$ define two Dirichlet characters. Let $\chi(n)=a_n$, which is the Dirichlet character of modular $4$ and $\hat\chi(n)=\alpha_n$ for $n=p_1^{r_1}p_2^{r_2}\cdots p_k^{r_k}$ can be written as
\begin{equation}
\hat\chi(n) = (-1)^{r_1+r_2+\cdots+r_k} \chi(n).
\end{equation}
Obviously, $\chi$ and $\hat\chi$ are multiplicative functions over $\intZ$.

\subsection{Special Values of The Twisted $L$-series}
In this section, we will explore the special values of the $L$-series $\zeta(M,1)$ and $\hat\zeta(M,1)$. As $\zeta(M,1)$ is the classical Dirichlet series, its special value at $1$ is well known:
\begin{lemma}[Dirichlet series \cite{book_fourier_analysis}*{Ch2} ]
Let $p_n$ be the $n$-th prime. Then
$$
\frac{\pi}{4} = \zeta(M,1) 
= \prod_{p\equiv 1\; (4)} \frac{1}{1-\frac{1}{p}} \cdot \prod_{p\equiv 3\; (4)} \frac{1}{1+\frac{1}{p}}
= \prod_{n=2}^\infty \frac{1}{1-\frac{(-1)^{\frac{p_n-1}{2}}}{p_n}}
= \prod_{n=2}^\infty \frac{1}{1-\frac{\sin \frac{\pi p_n}{2}}{p_n}}.
$$
\end{lemma}

With regard to the special value of $\hat\zeta$, we have:
\begin{lemma}[Euler's product formula of $\pi$]\label{lem:euler_pi_formula}
Let $p_n$ be the $n$-th prime. Then
$$
\frac{\pi}{2} = \hat\zeta(M,1)
= \prod_{p\equiv 1\; (4)} \frac{1}{1+\frac{1}{p}} \cdot \prod_{p\equiv 3\; (4)} \frac{1}{1-\frac{1}{p}}
= \prod_{n=2}^\infty \frac{1}{1+\frac{(-1)^{\frac{p_n-1}{2}}}{p_n}}
= \prod_{n=2}^\infty \frac{1}{1+\frac{\sin \frac{\pi p_n}{2}}{p_n}}.
$$
\end{lemma}
\begin{proof}
Consider the product
\begin{equation} \label{eq:convolution_prod_formula}
\begin{split}
&\Bigg(\prod_{p\equiv 1\; (4)} \frac{1}{1+\frac{1}{p}} \cdot \prod_{p\equiv 3\; (4)} \frac{1}{1-\frac{1}{p}} \Bigg)
\cdot \Bigg( \prod_{p\equiv 1\; (4)} \frac{1}{1-\frac{1}{p}} \cdot \prod_{p\equiv 3\; (4)} \frac{1}{1+\frac{1}{p}} \Bigg) \\
&= \prod_{p:\,\text{odd prime}} \frac{1}{1-\frac{1}{p^2}} = \sum_{n:\,\text{odd}} \frac{1}{n^2} = \frac{\pi^2}{8}.
\end{split}
\end{equation}
Since we know $\prod_{p\equiv 1\; (4)} \frac{1}{1-\frac{1}{p}} \cdot \prod_{p\equiv 3\; (4)} \frac{1}{1+\frac{1}{p}} = \sum_{k=1}^\infty {\spmode(-1)^{2k-1}} \frac{1}{2k-1} = \frac{\pi}{4}$, which implies the desired result.
\end{proof}

We can easily extend the argument and result in \eqref{eq:convolution_prod_formula} by raising the power of $1/p$ to $s$ as follows.
\begin{thm} [Functional Equation of $L$-series]  \label{cor:convolution_formula}
$$
\sum_{k=1}^\infty \frac{(-1)^{2k-1}}{(2k-1)^s} \cdot \sum_{k=1}^\infty \frac{\alpha_{2k-1}}{(2k-1)^s} 
= \sum_{k=1}^\infty \frac{1}{(2k-1)^{2s}}.
$$
Particularly, when $s=1$,
$$
\sum_{k=1}^\infty \frac{(-1)^{2k-1}}{2k-1} \cdot \sum_{k=1}^\infty \frac{\alpha_{2k-1}}{2k-1} 
= \sum_{k=1}^\infty \frac{1}{(2k-1)^2}.
$$
\end{thm}


\section{Main Results} \label{sec:main_result}
Now we turn to examine the geometric side of $M$. We will present our main identity in relating the twisted $L$-series $\hat\zeta(M,1)$ with the quarter circumference integral $\int_0^1 \frac{1}{\sqrt{1-x^2}} \,\d x$.
\subsection{Elementary Identities}
First, recall some elementary integrals in calculus:
\begin{lemma} \label{lem:int_sin^n&int_cos^n}
\begin{equation}
\begin{split}
\int_0^{\pi/2} \sin^nx \,\d x &= \frac{n-1}{n} \int_0^{\pi/2} \sin^{n-2}x \,\d x, \\
\int_0^{\pi/2} \cos^nx \,\d x &= \frac{n-1}{n} \int_0^{\pi/2} \cos^{n-2}x \,\d x.
\end{split}
\end{equation}
\end{lemma}
\begin{proof}
The proof only involves integration by parts, so we leave the details to the readers.
\end{proof}

Notice that $\int_0^{\pi/2} \cos^2 x \,\d x = \frac{\pi}{4}$ and $\sum_{k=0}^\infty (-1)^k \frac{1}{2k+1}=\frac{\pi}{4}$ numerically. In the following calculation, we illustrate a combinatorial transformation between the integral and the series.
\begin{lemma} \label{lem:int_cos^2}
$$
\int_0^{\pi/2} \cos^2 x \,\d x = \frac{\pi}{4} = \sum_{k=0}^\infty (-1)^k \frac{1}{2k+1}.
$$
\end{lemma}
\begin{proof}
By trigonometrical substitution:
$$
\int_0^{\pi/2} \cos^2 x \,\d x = \int_0^{\pi/2} \frac{1}{1+\tan^2 x} \,\d x
=\int_0^\infty \frac{1}{(1+u^2)^2} \,\d u,
$$
where $u=\tan x$. Compute
\begin{align*}
&\int_0^\infty \frac{1}{(1+u^2)^2} \,\d u = \int_0^1 \frac{1}{(1+u^2)^2} \,\d u 
+ \int_1^\infty \frac{1}{(1+u^2)^2} \,\d u \\
&= \int_0^1 \frac{1}{(1+u^2)^2} \,\d u + \int_0^1 \frac{v^2}{(v^2+1)^2} \,\d v
= \int_0^1 \frac{1}{(1+u^2)^2} \,\d u + \int_0^1 \frac{u^2}{(1+u^2)^2} \,\d v \\
&= \int_0^1 \frac{1}{1+u^2} \,\d u = \sum_{k=0}^\infty (-1)^k \frac{1}{2k+1}.
\end{align*}
by letting $v=1/u$. Hence, we are done.
\end{proof}

\begin{notation}[Double factorial]
Let $n\in\intN$ be an integer. If $n=2k$ is even, then
$$
n!! = \prod_{1\le i \le k} (2i) = n(n-2)\cdots 4\cdot 2;
$$
if $n=2k+1$ is odd, then
$$
n!! = \prod_{1\le i \le k+1} (2i-1) = n(n-2)\cdots 3\cdot 1.
$$
\end{notation}

By iterating Lemma \ref{lem:int_sin^n&int_cos^n}, we have the following identities:
\begin{thm} \label{cor:integ_formula_of_!!}
$$
\int_0^{\pi/2} \sin^{2k}x \,\d x  = \int_0^{\pi/2} \cos^{2k}x \,\d x
= \frac{1\cdot 3\cdots (2k-1)}{2\cdot 4\cdots (2k)} 
\cdot 2 \sum_{k=0}^\infty (-1)^k \frac{1}{2k+1}.
$$
In other words,
\begin{equation}
\frac{(2k-1)!!}{(2k)!!} = \frac{1\cdot 3\cdots (2k-1)}{2\cdot 4\cdots (2k)} 
= \frac{\int_0^{\pi/2} \sin^{2k}x \,\d x}{2 \sum_{k=0}^\infty {\scriptstyle(-1)^k} \frac{1}{2k+1}}
= \frac{\int_0^{\pi/2} \cos^{2k}x \,\d x}{2 \sum_{k=0}^\infty {\spmode(-1)^k} \frac{1}{2k+1}}.
\end{equation}
\end{thm}
\begin{proof}
This is a direct consequence of Lemma \ref{lem:int_sin^n&int_cos^n} and Lemma \ref{lem:int_cos^2}:
$$
\int_0^{\pi/2} \cos^{2k}x \,\d x = \frac{1\cdot 3\cdots (2k-1)}{4\cdot 6 \cdots (2k)} 
\cdot \int_0^{\pi/2} \cos^{2}x \,\d x = \frac{1\cdot 3\cdots (2k-1)}{2\cdot 4 \cdots (2k)} 
\cdot 2 \sum_{k=0}^\infty (-1)^k \frac{1}{2k+1}.
$$
\end{proof}


\subsection{Twisted $L$-series And The Real Period}
We are ready to present the proof of Theorem \ref{thm:main_result}. Recall a simple fact about the circumference integral of a circle:
\begin{lemma}\label{lem:integ_cir_of_circle}
$$
\int_{-1}^1 \frac{1}{\sqrt{1-x^2}} \, \d x = \pi.
$$
\end{lemma}
\begin{proof}
By the simple fact that $\int \frac{1}{\sqrt{1-x^2}} \, \d x = \sin^{-1} x$, we have $\int_{-1}^1 \frac{1}{\sqrt{1-x^2}} \, \d x= \sin^{-1} x\big|^1_{-1} = \pi/2 - (-\pi/2)=\pi$.
\end{proof}
Inspired by the identity, we aim to build a relation between the twisted $L$-series $\hat\zeta(1)$ and the integral $\int_0^1 \frac{1}{\sqrt{1-x^2}} \, \d x$. By detailed calculation, we can uncover a combinatorial transformation between the integral and the series.
\begin{thm}[Birch Type Theorem of A Rational Curve]
\begin{equation} \label{eq:L_id}
\hat\zeta(X,1) = \frac{1}{2} \int_{-1}^1 \frac{1}{\sqrt{1-x^2}} \, \d x
= \int_{0}^1 \frac{1}{\sqrt{1-x^2}} \, \d x.
\end{equation}
\end{thm}
\begin{proof}
By the binomial expansion, we have
\begin{align}
\notag S &= \int_{0}^1 \frac{1}{\sqrt{1-x^2}} \, \d x = 
\int_0^1 \sum_{k=0}^\infty C^{-1/2}_k (-x^2)^k \,\d x 
= \int_0^1 \sum_{k=0}^\infty \frac{1\cdot 3\cdots (2k-1)}{2\cdot 4\cdots (2k)} x^{2k} \,\d x  \\
\label{eq:period_integ_trig}&= \sum_{k=0}^\infty \frac{(2k-1)!!}{(2k)!!} \cdot \frac{1}{2k+1} 
= \frac{1}{A} \int_0^{\pi/2} \sum_{k=0}^\infty  \cos^{2k}x \cdot \frac{1}{2k+1}  \,\d x.
\end{align}
Note that we apply Theorem \ref{cor:integ_formula_of_!!} to the last equality and
$$
A=2 \sum_{k=0}^\infty (-1)^k \frac{1}{2k+1}.
$$ 
Let $f(t)=\sum_{k=0}^\infty \frac{t^{2k}}{2k+1}$. Then $f(t)\cdot t = \sum_{k=0}^\infty \frac{t^{2k+1}}{2k+1}$, which implies 
\begin{equation}
(f(t)\cdot t)' = \sum_{k=0}^\infty t^{2k}=\frac{1}{1-t^2} \implies
f(t) = \frac{1}{t} \int \frac{1}{1-t^2}\,\d t = \frac{1}{2}\ln \Big(\frac{1+t}{1-t}\Big) \cdot \frac{1}{t}
\end{equation}
when $|t|<1$. Apply this identity to \eqref{eq:period_integ_trig} and we then obtain
\begin{align*}
S &= \frac{1}{A} \int_0^{\pi/2} \frac{1}{2} \ln \Big(\frac{1+\cos x}{1-\cos x}\Big) \cdot \frac{1}{\cos x}  \,\d x
= \frac{1}{A} \int_0^{\pi/2} (-\ln \tan (x/2)) \frac{\sec^2 (x/2)}{1-\tan^2(x/2)} \,\d x \\
&= \frac{1}{A} \int_0^{1} (-\ln u) \frac{2}{1-u^2} \,\d u \quad\quad\text{by letting $u=\tan(x/2)$} \\
&= \frac{2}{A} \int_0^{1} (-\ln u) \sum_{k=0}^\infty u^{2k} \,\d u 
= \frac{2}{A} \sum_{k=0}^\infty  \int_0^{1} (-\ln u) u^{2k} \,\d u.
\end{align*}
With the aid of integration by parts, we have
\begin{align*}
\int_0^{1} (-\ln u) u^{2k} \,\d u &= \lim_{b\goto 0}(-\ln u)\frac{u^{2k+1}}{2k+1} \Bigg|_b^1 
+\int_0^1 \frac{1}{u}\cdot \frac{u^{2k+1}}{2k+1} \,\d u \\
&= \lim_{b\goto 0} \Big( \ln b \frac{b^{2k+1}}{2k+1} \Big) + \int_0^1 \frac{u^{2k}}{2k+1} \,\d u
= \frac{1}{(2k+1)^2}.
\end{align*}
Therefore, we have derived
$$
S = \frac{2}{A} \sum_{k=0}^\infty \frac{1}{(2k+1)^2}
= \frac{\sum_{k=0}^\infty \frac{1}{(2k+1)^2}}{\sum_{k=0}^\infty {\scriptstyle(-1)^k} \frac{1}{2k+1}}
= \sum_{k=0}^\infty \frac{\alpha_{2k+1}}{2k+1},
$$
by Theorem \ref{cor:convolution_formula}. The desired result follows.
\end{proof}

\section*{Acknowledgements}
We would like to specially thank Matt Burnham for providing the insight and proof of Lemma \ref{lem:euler_pi_formula}. 




\bib{ref_aabrandt&hansen_18}{article}{
  author={Aabrandt, Andreas},
  author={Hansen, Vagn Lundsgaard},
  title={The circle equation over finite fields},
  journal={Quaest. Math.},
  volume={41},
  date={2018},
  number={5},
  pages={665--674},
  issn={1607-3606},
  review={\MR {3836414}},
  doi={10.2989/16073606.2017.1395774},
}

\bib{ref_birch71}{article}{
  author={Birch, B. J.},
  title={Elliptic curves over $Q$: A progress report},
  conference={ title={1969 Number Theory Institute}, address={Proc. Sympos. Pure Math., Vol. XX, State Univ. New York, Stony Brook, N.Y.}, date={1969}, },
  book={ publisher={Amer. Math. Soc., Providence, R.I.}, },
  date={1971},
  pages={396--400},
  review={\MR {0314845}},
}

\bib{ref_bouw&wewers_17}{article}{
  author={Bouw, Irene I.},
  author={Wewers, Stefan},
  title={Computing $L$-functions and semistable reduction of superelliptic curves},
  journal={Glasg. Math. J.},
  volume={59},
  date={2017},
  number={1},
  pages={77--108},
  issn={0017-0895},
  review={\MR {3576328}},
  doi={10.1017/S0017089516000057},
}

\bib{ref_serre70}{article}{
  author={Serre, Jean-Pierre},
  title={Facteurs locaux des fonctions z\^{e}ta des variet\'{e}s alg\'{e}briques (d\'{e}finitions et conjectures)},
  language={French},
  conference={ title={S\'{e}minaire Delange-Pisot-Poitou. 11e ann\'{e}e: 1969/70. Th\'{e}orie des nombres. Fasc. 1: Expos\'{e}s 1 \`a 15; Fasc. 2: Expos\'{e}s 16 \`a 24}, },
  book={ publisher={Secr\'{e}tariat Math., Paris}, },
  date={1970},
  pages={15},
  review={\MR {3618526}},
}

\bib{ref_arithmetic_of_ellip_curve}{book}{
  author={Silverman, Joseph H.},
  title={The arithmetic of elliptic curves},
  series={Graduate Texts in Mathematics},
  volume={106},
  edition={2},
  publisher={Springer, Dordrecht},
  date={2009},
  pages={xx+513},
  isbn={978-0-387-09493-9},
  review={\MR {2514094 (2010i:11005)}},
  doi={10.1007/978-0-387-09494-6},
}

\bib{book_fourier_analysis}{book}{
  author={Stein, Elias M.},
  author={Shakarchi, Rami},
  title={Fourier analysis},
  series={Princeton Lectures in Analysis},
  volume={1},
  note={An introduction},
  publisher={Princeton University Press, Princeton, NJ},
  date={2003},
  pages={xvi+311},
  isbn={0-691-11384-X},
  review={\MR {1970295}},
}

\bib{ref_yih17}{article}{
  author={Sung, Yih},
  title={Rational points over finite fields on a family of higher genus curves and hypergeometric functions},
  journal={Taiwanese J. Math.},
  volume={21},
  date={2017},
  number={1},
  pages={55--79},
  issn={1027-5487},
  review={\MR {3613974}},
  doi={10.11650/tjm.21.2017.7724},
}



\begin{bibdiv}
\begin{biblist}

\bibselect{ref}

\end{biblist}
\end{bibdiv}
\end{document}